\documentclass[12pt]{article}

\usepackage{amsmath,amsthm,amsfonts,amssymb,amscd,cite}

\allowdisplaybreaks[4]

\newtheorem {lemma}{Lemma}
\newtheorem {theorem} {Theorem}

\newtheorem {proposition}{Proposition}

\begin{document}

\title{Laplacain eigenvalue distribution and diameter of graphs}

\author{Leyou Xu\footnote{E-mail: leyouxu@m.scnu.edu.cn}, Bo Zhou\footnote{Corresponding author. E-mail:
zhoubo@scnu.edu.cn}\\
School of Mathematical Sciences, South China Normal University\\
Guangzhou 510631, P.R. China}

\date{}
\maketitle

\begin{abstract}
Let $G$ be a connected graph on $n$ vertices with diameter $d$.
It is known  that if $2\le d\le n-2$, there are at most $n-d$ Laplacian eigenvalues in the interval $[n-d+2, n]$. In this paper, we show that  if $1\le d\le n-3$, there are at most $n-d+1$ Laplacian eigenvalues in the interval $[n-d+1, n]$. Moreover, we try to identify the connected graphs on $n$ vertices with diameter $d$, where $2\le d\le n-3$, such that there are at most $n-d$ Laplacian eigenvalues in the interval $[n-d+1, n]$.
\end{abstract}

\section{Introduction}

Let $G$ be a graph of order $n$. The Laplacian matrix of $G$ is $L(G)= D(G)-A(G)$, where $D(G)$ is the diagonal degree matrix and $A(G)$ is the
adjacency matrix $G$. The  eigenvalues of $L(G)$ are known as the Laplacian eigenvalues of $G$.
The Laplacian eigenvalues of $G$ lie in the interval $[0,n]$ \cite{Me,Moh}.
The distribution of
Laplacian eigenvalues in $[0,n]$ is a natural  problem, which is relevant due to
the many applications related to Laplacian matrices  \cite{HJT,Me,Mo}.
There do exist results that bound the number
of Laplacian eigenvalues in subintervals of $[0, n]$, see, e.g. \cite{CJT,GM,GT,HJT,Sin,ZZD}).
 %Jacobs, Oliveira and Trevisan \cite{JOT} showed that $m_G[0,2-\frac{2}{n})\ge \lceil\frac{n}{2}\rceil$ (or equivalently, $m_G[2-\frac{2}{n},n]\le \lfloor\frac{n}{2}\rfloor$) if
%$G$ is a tree. Here we note that $2-\frac{2}{n}$ is the average degree.
Generally, how the Laplacian eigenvalues are
distributed in the interval $[0, n]$ is a hard problem \cite{JOT} that is not well understood.

For a graph $G$ and a Laplacian eigenvalue $\mu$ of $G$ the multiplicity of $\mu$ is denoted by $m_G(\mu)$. For a graph $G$ of order $n$ and an interval $I\subseteq [0,n]$, the number of Laplacian eigenvalues of $G$ in $I$ is denoted by $m_GI$. Evidently, $m_G[0,n]=n$ for a graph $G$ of order $n$.
As pointed out by Jacobs, Oliveira and Trevisan in \cite{JOT},
it is also a hard problem because little is known about how the Laplacian eigenvalues are
distributed in the interval $[0, n]$. It is of interest to explore the relation between the distribution of Laplacian eigenvalues and the diameter of a graph. The following results are known.

% Let $G$ be a connected graph of order $n$ with diameter $d$. ??? showed that if  $d\ge 4$,  then $m_G(n-d+3, n]\le n-d-1$ and conjectured that if $d<n-1$, then  $m_G[n-d+2,n]\le n-d$. This was confirmed in \cite{XZ}.

\begin{theorem} \cite{GMS}
For any $n$-vertex connected graph $G$ with diameter $d\ge 1$, $m_G(2,n]\ge \lceil\frac{d}{2}\rceil$.
\end{theorem}

\begin{theorem}\label{AA} \cite{AhMS}
For any $n$-vertex connected graph $G$ with diameter $d$, where $d\ge 4$,  $m_G(n-d+3,n]\le n-d-1$.
\end{theorem}

\begin{theorem}\label{xz} \cite{XZ}
For any $n$-vertex connected graph $G$ with diameter $d$, where $2\le d\le n-2$,  $m_G[n-d+2,n]\le n-d$.
\end{theorem}

Note that Theorem \ref{xz} was conjectured by Ahanjideh et al. \cite{AhMS}.

In this paper, we show the following theorem.

\begin{theorem}\label{2xz}
For any $n$-vertex connected graph $G$ with diameter $d$, where $1\le d\le n-3$,  $m_G[n-d+1,n]\le n-d+1$.
\end{theorem}

%It is natural to consider whether $m_G[n-d+1,n]\le n-d$ for $n$-vertex connected graphs $G$ with diameter $d$.
Denote by $\mathfrak{G}(n,d)$ the class of graphs $G$ on $n$ vertices with diameter $d$ such that  $m_G[n-d+1, n]\le n-d$.   Note that $\mu_j(P_n)=4\sin^2\frac{(n-j)\pi}{2n}$ for $j=1,\dots,n$ \cite[p.~145]{AM}. So, if $n\ge 4$, then $\mu_2(P_n)\ge 2$, implying that $m_{P_n}[2,n]\ge 2$.
If $n\ge 10$, then for any $n$-vertex connected graph $G$ with diameter $n-2$, $m_G[3,n]\ge 3$. This is because
 \[
\mu_3(G)\ge\mu_{3}(P_{n-1})=4\sin^2\frac{(n-4)\pi}{2(n-1)}\ge 3
\]
by Lemma \ref{addedges} below. So, to determine  the graph class  $\mathfrak{G}(n,d)$, it is necessary that $1\le d\le n-3$.
Evidently, $m_{K_n}(n)=n-1$ if $n\ge 2$ and so $\mathfrak{G}(n,1)=\{K_n\}$.
We show that $\mathfrak{G}(n,d)$ contains all graphs with $d=2,3$, and those graphs with some diametral path (a diametral path of a graph is a shortest path whose length is equal to the diameter of the graph) $P$ such that there are at least two vertices outside $P$ with at most two neighbors on $P$, where $4\le d\le n-3$.
We also construct a class of
$n$-vertex graphs with diameter $d$ that is not in $\mathfrak{G}(n,d)$
for $4\le d\le n-3$.

\section{Preliminaries}

For a graph $G$, we denote by $V(G)$ the vertex and $E(G)$ the edge set.
Let $H_1\cup H_2$ be the disjoint union of graphs $H_1$ and $H_2$. The disjoint union of $k$ copies of a graph $G$ is denoted by $kG$.
As usual, denote by $P_n$  the path, and $K_n$ the complete graph, of order $n$.

For a graph $G$ with $E_1\subseteq E(G)$, denote by $G-E_1$ the subgraph of $G$ obtained from $G$ by deleting all edges in $E_1$.
Particularly, if $E_1=\{e\}$, then we write  $G-e$ for $G-\{e\}$.
If $G'=G-E_1$ with $E_1\subseteq E(G)$, then $G=G'+E_1$ and it as $G'+\{e\}$ if $E_1=\{e\}$.

For $v\in V(G)$, denote by $N_G(v)$ the neighborhood of $v$ in $G$, and  $\delta_G(v):=|N_G(v)|$ denotes the degree of $v$ in $G$. Let $H$ be subgraph of $G$ and $u\in V(G)\setminus V(H)$, let $\delta_H(u)=|N_G(u)\cap V(H)|$.

%We also need the following two types of interlacing theorem or inclusion principle.

\begin{lemma}\label{interlacing}\cite[Theorem 4.3.28]{HJ}
If $M$ is  a Hermitian matrix of order $n$ and $B$ is its principal submatrix of order $p$, then $\rho_{n-p+i}(M)\le\rho_i(B)\le \rho_{i}(M)$ for $i=1,\dots,p$.
\end{lemma}

\begin{lemma}\cite[Theorem 1.3]{So} \label{cw}
Let $A$ and $B$ be Hermitian matrices of order $n$. Then for $1\le i,j\le n$ with $i+j-1\le n$,
$\rho_{i+j-1}(A+B)\le \rho_i(A)+\rho_j(B)$ with equality if and only if there exists a nonzero vector $\mathbf{x}$ such that $\rho_{i+j-1}(A+B)\mathbf{x}=(A+B)\mathbf{x}$, $\rho_i(A)\mathbf{x}=A\mathbf{x}$ and $\rho_j(B)\mathbf{x}=B\mathbf{x}$.
\end{lemma}

\begin{lemma}\label{addedges} \cite[Theorem 3.2]{Moh}
If $G$ is a graph with $e\in E(G)$, then
\[
\mu_1(G)\ge\mu_1(G-e)\ge \mu_2(G)\ge\dots\ge \mu_{n-1}(G-e)\ge \mu_n(G)=\mu_{n}(G-e)=0.
\]
\end{lemma}

%\begin{lemma}\label{mu1}\cite{GM}
%For a connected $n$-vertex graph $G$ with largest degree $\Delta$, $\mu_1(G)\ge \Delta+1$ with equality if and only if $\Delta=n-1$.
%\end{lemma}

Denote by $\overline{G}$ the complement of a graph $G$.

\begin{lemma}\label{n-}\cite[Theorem 3.6]{Moh}
If $G$ is a graph of order $n\ge 2$, then $\mu_i(G)=n-\mu_{n-i}(\overline{G})$ for $i=1,\dots,n-1$.
\end{lemma}

It is known that the number of components of a graph is equal to the multiplicity of $0$ as a Laplacian eigenvalue \cite{Me}.

For an $n\times n$ Hermitian matrix $M$, $\rho_i(M)$ denotes its $i$-th largest eigenvalue of $M$ and $\sigma(M)=\{\rho_i(M):i=1,\dots,n\}$ is the spectrum of $M$.
For convenience, if $\rho$ is an eigenvalue of $M$ with multiplicity $s\ge 2$, then we write it as $\rho^{[s]}$ in $\sigma(M)$. For a graph $G$, let $\sigma_L(G)=\sigma(L(G))$.

Let $G$ be a graph with a partition $\pi$: $V(G)=V_1\cup \dots\cup V_m$. Then $L(G)$ may be partitioned into a block matrix, where   $L_{ij}$ denotes the block with rows corresponding to vertices in $V_i$ and columns corresponding to vertices in $V_j$ for $1\le i, j\le m$. Denote by $b_{ij}$ the average row sums of $L_{ij}$ for $1\le i, j\le m$. Then the $m\times m$ matrix  $B=(b_{ij})$ is called
the quotient matrix of $L(G)$ with respect to $\pi$. If $L_{ij}$ has constant row sum for all $1\le i, j\le m$, then we say $\pi$ is an equitable partition.
The following lemma is an immediate result of \cite[Lemma 2.3.1]{BH}.

\begin{lemma}\label{quo}
For a graph $G$, if $B$ is the quotient matrix of $L(G)$ with respect to an equitable partition, then $\sigma(B)$ is contained in $\sigma_L(G)$.
\end{lemma}

A double star $D_{n,a}$ is a tree with diameter $3$ in which there are $a$ and $n-2-a$ pendant edges at its non-pendant vertices, where $1\le a\le n-3$.

We need a result due to Doob  stating that for a tree $T$ on $n$ vertices with diameter $d\ge 1$,
$\mu_{n-1}(T)\le 2\left(1-\cos\frac{\pi}{d+1}\right)$, see \cite[p.~187]{CDS} or
\cite[Corollary 4.4]{GMS}.

\begin{lemma} \label{double}
Let $G=D_{n,a}$, where $n\ge 5$, $\mu_2(G)>2$ and $\mu_3(G)=1$.
\end{lemma}

\begin{proof}
As $D_{5,1}$ is an edge induced subgraph, we have by Lemma \ref{addedges} that $\mu_2(G)\ge \mu_2(D_{5,1})\approx 2.311>2$.

As $I_n-L(G)$ has $a$ and $n-2-a$ equal rows, respectively, $1$ is a Laplacian eigenvalue of $G$ with multiplicity at least $(a-1)+(n-2-a-1)=n-4$.
By the result due to Doob as mentioned above,  $\mu_{n-1}(G)\le  2\left(1-\cos\frac{\pi}{4}\right)<1$. It follows that $\mu_3(G)=1$.
\end{proof}

For integers $n$, $d$ and $t$ with $2\le d\le n-2$ and $2\le t\le d$,  $G_{n,d,t}$ denotes the graph  obtained from a path $P_{d+1}:=v_1\dots v_{d+1}$ and a complete graph $K_{n-d-1}$ such that they are vertex disjoint by adding all edges connecting vertices of $K_{n-d-1}$ and vertices $v_{t-1}$, $v_t$ and $v_{t+1}$.

Two matrices $A$ and $B$ are said to be permutationally similar if
$B=P^\top A P$ for some permutation matrix $P$. That is, $B$ is obtainable from $A$ by simultaneous
permutations of its rows and columns.

\section{Proof of Theorem \ref{2xz}}

%\begin{theorem}\label{2xz}
%Let $G$ be a connected $n$-vertex graph with diameter $d$, where $1\le d\le n-3$. Then $m_G[n-d+1,n]\le n-d+1$.
%\end{theorem}
\begin{proof}[Proof of Theorem \ref{2xz}]
It is trivial for $d=1$. Suppose the following that $d\ge 2$.
Let $P:=v_1\dots v_{d+1}$ be the diametral path of $G$.
As $d\le n-3$, there are at least two vertices outside $P$, say $u$ and $v$.
Let $H$ be the subgraph induced by $V(P)\cup \{u,v\}$ and $B$ the principal submatrix of $L(G)$ corresponding to vertices of $H$.
Let $B=L(H)+M$. Then $M$ is a diagonal matrix whose  diagonal entry corresponding to vertex $z$ is
$\delta_G(z)-\delta_H(z)$ for $z\in V(H)$. Obviously,  $\rho_1(M)\le n-|V(H)|=n-d-3$.
 If $\mu_5(H)<4$, then one gets  by Lemmas \ref{interlacing} and \ref{cw} that
 \[
\mu_{n-d+2}(G)=\rho_{n-(d+3)+5}(L(G))\le \rho_5(B)\le \mu_5(H)+\rho_1(M)<n-d+1,
\]
so $m_G[n-d+1,n]\le n-d+1$. Thus, it suffices to show that $\mu_5(H)<4$.

As $P$ is the diametral path, $u$ ($v$, respectively) is adjacent to at most three consecutive vertices on $P$.
If $\delta_P(u)\le 2$ ($\delta_P(v)\le 2$, respectively), $3-\delta_P(u)$ ($3-\delta_P(v)$, respectively) edges may be added to join $u$ ($v$, resepctively) and $3-\delta_P(u)$ ($3-\delta_P(v)$, respectively) vertices such that $u$ ($v$, respectively) is adjacent to exactly three consecutive vertices on $P$.
By Lemma \ref{addedges}, we may assume that $\delta_P(u)=\delta_P(v)=3$.
Let $v_{t-1},v_t,v_{t+1}$ ($v_{r-1},v_r,v_{r+1}$, respectively) be three neighbors of $u$ ($v$, resepctively).
Assume that $t\le r$.

\noindent
{\bf Case 1.} $t=r$.

Let $H'=H+uv$ if $uv\notin E(G)$ and $H'=H$ otherwise. Then \[
L(H')=\begin{pmatrix}
L(P)&O_{(d+1)\times 2}\\
O_{2\times (d+1)}&L(P_2)
\end{pmatrix}+M,
\]
where $M=(m_{ij})_{(d+3)\times (d+3)}$ with \[
m_{ij}=\begin{cases}
3&\mbox{ if }i=j\in\{d+2,d+3\},\\
2&\mbox{ if }i=j\in\{t-1,t,t+1\},\\
-1&\mbox{ if }\{i,j\}\in \{\{p,q\}:p=t-1,t,t+1,q=d+2,d+3 \},\\
0&\mbox{ otherwise.}
\end{cases}
\]
Let $R=L(H')-M$. Then $\rho_1(R)=\mu_1(P)$.
As $M$ is  permutationally similar to $L(K_{2,3}\cup  (d-2)K_1)$, one gets
 $\rho_5(M)=0$.
By Lemma \ref{cw}, we have \[
\mu_5(H')\le \rho_1(R)+\rho_5(M)=\mu_1(P)<4,
\]
so $\mu_5(H)\le \mu_5(H')<4$ by Lemma \ref{addedges}.

\noindent
{\bf Case 2.} $t=r-1$.

Let $H'=H+uv$ if $uv\notin E(G)$ and $H'=H$ otherwise.
It suffices to show that $\mu_5(H')<4$ by Lemma \ref{addedges}.
Let $u_i=v_i$ for $i=1,\dots,t+1$, $u_{t+2}=v$, $u_{i+1}=v_i$ for $i=t+2,\dots, d+1$, and $u_{d+3}=u$.
It is easy to see that $H'$ is obtainable from $G_{d+3,d+1,t}$ by adding edges $u_tu_{t+2},$ $u_{t+1}u_{t+3}$ and $u_{t+2}u_{d+3}$.
Then, under the ordering $u_1, \dots, u_{d+3}$,
\[
L(H')=L(G_{d+3,d+1,t})+M,
\]
where $M=(m_{ij})_{(d+3)\times (d+3)}$ with \[
m_{ij}=\begin{cases}
2&\mbox{ if }i=j=t+2,\\
1&\mbox{ if }i=j\in \{t,t+1,t+3,d+3 \},\\
-1&\mbox{ if }\{i,j\}\in \{\{t,t+2 \},\{t+1,t+3\}, \{t+2,d+3\} \},\\
0&\mbox{ otherwise.}
\end{cases}
\]
Note that $M$ is  permutationally similar to $L(P_3\cup P_2\cup  (d-2)K_1)$, one gets
 $\rho_4(M)=0$.
By \cite[Lemma 2.6]{XZ}, $\mu_2(G_{d+3,d+1,t})=4$.
By Lemma \ref{cw}, \[
\mu_5(H')\le \mu_2(G_{d+3,d+1,t})+\rho_4(M)= \mu_2(G_{d+3,d+1,t})=4.
\]

Suppose that $\mu_5(H')=4$. By Lemma \ref{cw}, there exists a nonzero vertor $\mathbf{x}$ such that $M\mathbf{x}=0$ and $L(H')\mathbf{x}=4\mathbf{x}=L(G_{d+3,d+1,t})\mathbf{x}$.
Let $x_i=x_{u_i}$ for $i=1,\dots,d+3$. From $M\mathbf{x}=0$ at $u_t$, $u_{t+1}$ and $u_{t+2}$, we have $x_t=x_{t+2}=x_{d+3}$ and $x_{t+1}=x_{t+3}$. From  $L(H')\mathbf{x}=4\mathbf{x}$ at $u_{t+2}$, one gets \[
2x_{t+2}-2x_{t+1}=4x_{t+2},
\]
so $x_{t+1}=-x_{t+2}=-x_t$.

From $L(H')=4\mathbf{x}$ at $u_1$, we have $x_1-x_2=4x_1$, so $x_2=-3x_1$. Suppose that $x_j=(-1)^{j-1}(2j-1)x_1$ for each $j\le i$ with $i=2,\dots,t-2$.
From $L(H')=4\mathbf{x}$ at $u_i$, we have $2x_i-x_{i-1}-x_{i+1}=4x_i$, so $x_{i+1}=-2x_i-x_{i-1}=(-1)^{i}(2i+1)x_1$. This shows that
\[
x_i=(-1)^{i-1}(2i-1)x_1 \mbox{ for } i=2,\dots,t-1.
\]
From  $L(H')=4\mathbf{x}$ at $u_{t-1}$, we have $3x_{t-1}-2x_t-x_{t-2}=4x_{t-1}$,
so $x_{t}=-(x_{t-1}+x_{t-2})/2=(-1)^tx_1$.

Similarly, we have $x_i=(-1)^{d+2-i}(2(d+3-i)-1)x_{d+2}$ for $i=t+1,\dots,d+2$.
From $x_{t+1}=x_{t+3}$, one gets $(-1)^{d-1}(2(d-t)-1)x_{d+2}=x_{t+3}=x_{t+1}=(-1)^{d+1}(2(d-t+2)-1)x_{d+2}$, so $x_{d+2}=0$.
Thus $x_t=-x_{t+1}=0$, implying that  $x_1=(-1)^tx_t=0$. It thus follows that $\mathbf{x}=0$, a contradiction.
Therefore $\mu_5(H')<4$.

\noindent
{\bf Case 3.} $t=r-2$.

As $P$ is a diametral path, $u$ is not adjacent to $v$.
Let $u_i=v_i$ for $i=1,\dots,t$, $u_{t+1}=u$, $u_{t+2}=v_{t+1}$, $u_{t+3}=v$ and $u_{i+2}=v_i$ for $i=t+2,\dots,d+1$.
Evidently, $H$ is obtainable from the path $u_1\dots u_{d+3}$ by adding edges $u_ju_{j+2}$ for $j=t-1,t,t+2,t+3$. %$u_{t-1}u_{t+1}$, $u_{t}u_{t+2}$, $u_{t+2}u_{t+4}$ and $u_{t+3}u_{t+5}$
Then \[
L(H)=L(P_{d+3})+M,
\]
where $M=(m_{ij})_{(d+3)\times (d+3)}$ with \[
m_{ij}=\begin{cases}
2&\mbox{ if }i=j=t+2,\\
1&\mbox{ if }i=j\in \{t-1,t,t+3\},\\
-1&\mbox{ if }\{i,j\}\in \{\{p,p+2\}:p=t-1,t,t+2,t+3 \},\\
0&\mbox{ otherwise.}
\end{cases}
\]
Note that $M$ is  permutationally similar to $L(P_3\cup 2P_2\cup  (d-4)K_1)$, one gets
$\rho_5(M)=0$.
By Lemma \ref{cw}, \[
\mu_5(H)\le \mu_1(P_{d+3})+\rho_5(M)=\mu_1(P_{d+3})<4,
\]
as desired.

\noindent
{\bf Case 4.} $t\le r-3$.

As $P$ is a diametral path, $u$ is not adjacent to $v$. Let $u_i=v_i$ for $i=1,\dots,t$, $u_{t+1}=u$, $u_{i+1}=v_i$ for $i=t-1,\dots,r$, $u_{r+2}=v$, $u_{i+2}=v_i$ for $i=r+1,\dots,d+1$. Evidently, $H$ is obtainable from $u_1\dots u_{d+3}$ by adding edges $u_ju_{j+2}$ for $j=t-1,t,r-1,r$.
Then \[
L(H)=L(P_{d+3})+M,
\]
where $M=(m_{ij})_{(d+3)\times (d+3)}$ with \[
m_{ij}=\begin{cases}
1&\mbox{ if }i=j\in \{t-1,t,t+1,t+2,r-1,r,r+1,r+2\},\\
-1&\mbox{ if }\{i,j\}\in \{\{p,p+2\}:p=t-1,t,r-1,r \},\\
0&\mbox{ otherwise.}
\end{cases}
\]
Note that $M$ is  permutationally similar to $L(4P_2\cup  (d-5)K_1)$, one gets
 $\rho_5(M)=0$. By Lemma \ref{cw},\[
\mu_5(H)\le \mu_1(P_{d+3})+\rho_5(M)<4,
\]
as desired.
\end{proof}

\section{Graph class $\mathfrak{G}(n,d)$}

In this section, we show that if $\mathfrak{G}(n,d)$ is actually the class of $n$-vertex graphs with diameter $d$ and determine the graphs $G$ in $\mathfrak{G}(n,d)$ such that $m_G[n-d+1,n]=n-d$ if $d=2,3$.
We  give a sufficient
condition such that $G\in \mathfrak{G}(n,d)$ and construct a class of
$n$-vertex graphs with diameter $d$ that is not in $\mathfrak{G}(n,d)$
for $4\le d\le n-3$.

\begin{theorem}\label{d2}
Let $G$ be an $n$-vertex graph with diameter two. Then $m_G[n-1,n]\le n-2$ with equality if and only if $G\cong K_n-\{vv_i:i=1,\dots,s\}$  for some $s=1,\dots,n-2$, where $v,v_1,\dots,v_s\in V(K_n)$.
\end{theorem}

\begin{proof}
If $G\cong K_n-\{vv_i:i=1,\dots,s\}$, then  $\overline{G}\cong K_{1,s}\cup (n-s-1)K_1$
$\sigma_L(\overline{G})=\{s+1,1^{[s-1]},0^{[n-s]} \}$, so by Lemma \ref{n-}, we have
$\sigma_L(G)=\{n^{[n-s-1]},n-1^{[s-1]},n-s-1,0 \}$, implying that $m_G[n-1,n]= n-2$.

Suppose that $G\ncong K_n-\{vv_i:i=1,\dots,s\}$ for any $s=1, \dots, n-2$. It suffices to show that
$m_G[n-1,n]< n-2$, or equivalently, $\mu_{n-2}(G)<n-1$.
As $G\ncong K_n-\{vv_i: i=1,\dots,s\}$ for any $s=1, \dots, n-2$, $G$ is a spanning subgraph of $H:=K_n-vv_1-e$ for some $e\in E(K_n-v-v_1)$.
As $\overline{H}\cong 2K_2\cup (n-4)K_1$ with $\sigma_L(\overline{H})=\{2^{[2]},0^{[n-2]}\}$, we have $\sigma_L(H)=\{n^{[n-3]}, n-2^{[n-2]}, 0\}$ by Lemma \ref{n-}, so $\mu_{n-1}(H)=n-2$.
Now, by Lemma \ref{addedges}, $\mu_{n-2}(G)\le\mu_{n-2}(H)=n-2<n-1$.
\end{proof}

For $n\ge 5$, let $G_{n,3}=G_{n,3,3}$.

\begin{lemma}\label{G123}
Let $G=G_{n,3}$ with a diametral path $P=v_1v_2v_3v_4$. Let $u$ and $v$ be two distinct vertices outside $P$. Then $\mu_{n-3}(G-v_3u)<n-2$, $\mu_{n-3}(G-v_2u-v_4u)<n-2$ for $u\in V(G)\setminus V(P)$,
and $\mu_{n-3}(G-v_2u-v_4v)<n-2$ for $\{u,v\}\subseteq V(G)\setminus V(P)$.
\end{lemma}

\begin{proof}
If $n=5$, it follows by a direct calculation that $\mu_2(G-v_3u)\approx 2.6889<3$ and $\mu_2(G-v_2u-v_4u)\approx 2.311<3$.

Suppose next that $n\ge 6$.
Let $G_1=\overline{G-v_3u}$,  $G_2=\overline{G-v_2u-v_4u}$ and $G_3=\overline{G-v_2u-v_4v}$.
It suffices to show that $\mu_3(G_i)>2$ by Lemma \ref{n-}.

Note that in  $G_1$, $\{uv_1,v_1v_3, uv_3, v_2v_4, v_4v_1\}$ induces a graph $G_1'$ obtained from $K_3$ and $K_2$ by adding an edge between them. Lemma \ref{addedges} implies that
$\mu_3(G)\ge \mu_3(G_1')\approx 2.311>2$.

Note that in $G_2$, $\{uv_2,uv_4, uv_1, v_2v_4, v_1v_3,v_1v_4\}$ induces $G_2'$. Lemma \ref{addedges} implies that
$\mu_3(G)\ge \mu_3(G_2')\approx 2.689 >2$.

Note that in $G_3$, $\{uv_2,uv_1, v_2v_4, v_1v_4, v_1v,v_4v\}$ induces $\overline{P_5}$. Lemma \ref{addedges} implies that
$\mu_3(G)\ge \mu_3(\overline{P_5})=5-4\sin^2\frac{3\pi}{10}>2$.
\end{proof}

For positive integers $s\le n-4$, denote by $G_{n,3}^{(2,s)}$ ($G_{n,3}^{(4,s)}$, respectively) be the graph obtained from $G_{n,3}$ by removing $s$ edges joining $v_2$ ($v_4$, respectively) and vertices outside the diametral path $v_1\dots v_4$, where $\delta_G(v_1)=1$.

%
%Let $s\le n-4$ be a nonnegative integer and $S$ be the set of vertices of degree $n-2$ in $G_{n,3}$ with $|S|=s$.
%Let \[
%G_{n,3}^{(2,s)}=G_{n,3}-\{v_2u:u\in S\}\mbox{ and }G_{n,3}^{(4,s)}=G_{n,3}-\{v_4u:u\in S\}.
%\]
%Particularly, if $s=0$, then $G_{n,3}^{(2,s)}=G_{n,3}^{(4,s)}=G_{n,3}$.

\begin{lemma}\label{h24}
(i) $m_{G_{n,3}}[n-2,n]=n-3$. \\
(ii) For  $1\le s\le n-4$, $\mu_{G_{n,3}^{(2,s)}}[n-2,n]=\mu_{G_{n,3}^{(4,s)}}[n-2,n]=n-3$.

\end{lemma}
\begin{proof}
By Lemmas \ref{n-} and \ref{double}, one gets $\mu_{n-2}(G_{n,3})<n-2$ and $\mu_{n-3}(G_{n,3})=n-1$,
so $m_{G_{n,3}}[n-2,n]=n-3$. This proves Item (i).

Next, we prove Item (ii). Denote by $S$ the set of vertices so that the removal of all edges $v_2u$ ($v_4u$, respectively) with $u\in S$ from $G_{n,3}$ with diametral path $v_1\dots v_4$ yields $G_{n,3}^{(2,s)}$ ($G_{n,3}^{(4,s)}$. Then $|S|=s$. Let $S'=V(G_{n,3})\setminus(S\cup \{v_1,v_2,v_4\})$.
%
%Note that both $G_{n,3}^{(2,s)}$ and $G_{n,3}^{(4,s)}$ are spanning subgraphs of $G_{n,3}$.
By Lemma
\ref{addedges}, one has $\mu_{n-2}(G_{n,3}^{(2,s)}),\mu_{n-2}(G_{n,3}^{(4,s)})\le \mu_{n-2}(G_{n,3})<n-2$.
So it suffices to show that $\mu_{n-3}(H)\ge n-2$ for $H=G_{n,3}^{(4,s)}, G_{n,3}^{(4,s)}$.

Let $H_1=\overline{G_{n,3}^{(2,s)}}$ and $H_2=\overline{G_{n,3}^{(4,s)}}$.
Let  $V(H_1)=\{v_2\}\cup (S\cup \{v_4\})\cup \{v_1\}\cup S'$ and $V(H_2)=\{v_2\}\cup \{v_4\}\cup S\cup \{v_1\}\cup S'$. With respect to the above partitions, $L(H_1)$ and $L(H_2)$ have  quotient matrices $B_1$ and $B_2$, respectively, where
\[
B_1=\begin{pmatrix}
s+1&-s-1&0&0\\
-1&2&-1&0\\
0&-s-1&n-2&-n+s+3\\
0&0&-1&1
\end{pmatrix},\]
and
\[
B_2=\begin{pmatrix}
1&-1&0&0&0\\
-1&s+2&-s&-1&0\\
0&-1&2&-1&0\\
0&-1&-s&n-2&-n+s+3\\
0&0&0&-1&1
\end{pmatrix}.
\]
As both $B_1$ and $B_2$ has all row sums to be $0$, $0$ is an eigenvalue, so we may assume that
 $\det(xI_4-B_1)=xf(x)$ and $\det(xI_5-B_2)=xg(x)$.
Note that the partitions are equitable.  So the roots of $f(x)=0$ ($g(x)=0$, respectively) are Laplacian eigenvalues of $H_1$ ($H_2$, respectively).
By direct calculations,
\[
f(x)=x^3-(n+2+s)x^2+((s+3)n-2)x-(s+1)n
\]
and
\[
g(x)=x^4-(n+4+s)x^3+((s+5)n+s+2)x^2-((2s+7)n-s-4)x+(s+2)n.
\]
As $f(0)=-(s+1)n<0$ and $f(1)=n-s-3>0$, $g(0)=(s+2)n>0$, $g(1)=-n+s+3$ and $g(2)=2s(n-2)>0$,
$f(x)$ has a root $a$ with $0<a<1$, and $g(x)$ has roots $b$ and $c$ with $0<b<1$ and $1<c<2$.
By Lemmas \ref{quo} and \ref{n-}, $n-a$ is a Laplacian eigenvalue of $G_{n,3}^{(2,s)}$, $n-b$ and $n-c$ are Laplacian eigenvalues of $G_{n,3}^{(4,s)}$.

Note that $(n-1)I_n-L(G_{n,3}^{(2,s)})$ and $(n-2)I_n-L(G_{n,3}^{(2,s)})$ has $n-s-3$ and $s+1$ equal rows, respectively. Then $n-1$ and $n-2$ are  eigenvalues of $L(G_{n,3}^{(2,s)})$ with multiplicity $n-s-4$ and $s$, respectively. Recall that $n-a$ is a Laplacian eigenvalue of $G_{n,3}^{(2,s)}$. So $\mu_{n-3}(G_{n,3}^{(2,s)})\ge n-2$.

As $(n-1)I_n-L(G_{n,3}^{(4,s)})$ and $(n-2)I_n-L(G_{n,3}^{(4,s)})$ has $n-s-3$ and $s$ equal rows, $n-1$ and $n-2$ are Laplacian eigenvalues of $G_{n,3}^{(4,s)}$ with multiplicity $n-s-4$ and $s-1$, respectively.
So $\mu_{n-3}(G_{n,3}^{(4,s)})\ge n-2$.
\end{proof}

Let $n\ge 6$ and $a$ be an integer with $1\le a\le n-5$.
Let $G^{n,a}$ be a graph obtained from a path $P_4:=v_1v_2v_3v_4$ and a complete graph $K_{n-4}$ by adding edges connecting vertices of $K_{n-4}$ and vertices $v_2,v_3$, and adding edges between $a$ vertices of $K_{n-4}$ and $v_1$ and the remaining $n-4-a$ vertices of $K_{n-4}$ and $v_4$.

\begin{lemma}\label{gs}
Let $n$ and $a$ be integers with $1\le a\le \frac{n}{2}-2$. Let $G=G^{n,a}$ with a diametral path $P=v_1\dots v_4$.
Let $u$ ($w$, respectively) be a neighbor of $v_1$ ($v_4$, respectively) outside $P$ in $G$.
Then $\mu_{n-3}(G-v_iu)<n-2$ and $\mu_{n-3}(G-v_iw)<n-2$ for $i=2,3$.
\end{lemma}

\begin{proof}
Let $H_1=\overline{G-v_2u}$ and $H_2=\overline{G-v_2w}$.
%It suffices to show that $\mu_3(H_i)>2$.

Note that $H_1$ contains $H_1'$ a an edge induced subgraph consisting of a $C_3=uv_2v_4u$ and a star with $3$ edges $v_1v_4, v_1v_3, v_1w$. By Lemma \ref{addedges}, $\mu_3(H_1)\ge \mu_3 (H_1')=3>2$. So
$\mu_{n-3}(G-v_2u)<n-2$. Similarly, $\mu_{n-3}(G-v_3w)<n-2$.

Note that in $H_2$, $\{v_2v_4, v_1v_4, wv_2, wv_1, uv_4, v_1v_3\}$ induces the graph $H_2'$ consisting of a $C_4$ with a pendant edge attached at each of two adjacent vertices. By Lemma \ref{addedges}, $\mu_3(H_2)\ge \mu_3 (H_2')\approx 2.529>2$. So  $\mu_{n-3}(G-v_2w)<n-2$. Similarly, $\mu_{n-3}(G-v_3u)<n-2$.
\end{proof}

Let $a$, $b$ and $n$ be integers with $a,b\ge 1$ and $a+b\le n-4$. Let
$G^{n,a,b}$ be the graph obtained from a path $P_4=v_1v_2v_3v_4$ and a complete graph $K_{n-4}$ by adding edges connecting vertices of $K_{n-4}$ and vertices $v_2$, $v_3$, adding edges between $a$ vertices of $K_{n-4}$ and $v_1$ and other $b$ vertices of $K_{n-4}$ and $v_4$.
%Particularly, if $a+b=n-4$, then $G^{n,a,b}\cong G^{n,a}$.

\begin{lemma}\label{g3ab} (i) For $1\le a\le \frac{n}{2}-2$,
$m_{G^{n,a}}[n-2, n]=n-3$.\\
(ii) For $1\le a\le b$ and $a+b\le n-5$,
$m_{G^{n,a,b}}[n-2, n]=n-3$.
\end{lemma}

%\begin{lemma}\label{gx3ab} (i) For $1\le a\le \frac{n}{2}-2$,
%$\mu_{n-2}(G^{n,a})<n-2$ and $\mu_{n-3}(G^{n,a})=n-1$.\\
%(ii) For $1\le a\le b$ and $a+b\le n-5$, then
%$n-2\le \mu_{n-3}(G^{n,a,b})< n-1$.
%\end{lemma}

\begin{proof}
By Lemmas \ref{n-} and \ref{double}, $\mu_{n-2}(G^{n,a})<n-2$ and $\mu_{n-3}(G^{n,a})=n-1$, so Item (i) follows.

Next, we prove Item (ii).
Let $H=\overline{G^{n,a,b}}$.

As $H$ contains $H'$ as an edge induced subgraph, obtained by attaching two pendant edges at each of two adjacent vertices of a $C_3$, we have by Lemma \ref{addedges} that
 $\mu_2(H)\ge \mu_2(H')\approx 4.4142>4$.

As $I_n-L(H)$ has $a+1$ and $b+1$ equal rows, $1$ is a Laplacian eigenvalue of $H$ with multiplicity at least $a+b$. Similarly, $2$ is a Laplacian eigenvalue of $H$ with  multiplicity at least $n-5-(a+b)$.

Let $U_1=\{y\in N_H(v_1):\delta_{H}(y)=1\}$ and $U_2=\{y\in N_H(v_4):\delta_{H}(y)=1\}$. Then $|U_1|=a+1$ and $|U_2|=b+1$. Let $U_3=V(H)\setminus (U_1\cup U_2\cup \{v_1,v_4\})$.
The quotient matrix $B$ of $L(H)$ with respect to the
partition $V(H)=U_1\cup \{v_1\}\cup U_3\cup \{v_4\}\cup U_2$ is
\[
B=\begin{pmatrix}
1&-1&0&0&0\\
-(b+1)&n-a-2&-(n-4-a-b)&-1&0\\
0&-1&2&-1&0\\
0&-1&-(n-4-a-b)&n-b-2&-(a+1)\\
0&0&0&-1&1
\end{pmatrix}.
\]
Note the this partition is equitable.  So the roots of $f(x)=0$ are Laplacian eigenvalues of $H$, where
 $\det(xI_5-B)=xf(x)$.
As
\[
f(0)=(n-4-a-b)n>0,\mbox{ }f(1)=-(a+1)(b+1)<0,
\]
and
\[
f(2)=(n-4-a-b)(n-2)>0,
\]
$f(x)$ has roots $\alpha$ and $\beta$ with $0<\alpha<1$ and $1<\beta<2$.

Therefore, if $a+b=n-5$, then $\mu_2(H)>4$ and
$1<\mu_3(H)=\beta<2$, and if $a+b<n-5$, then  $\mu_2(H)>4$ and  $\mu_3(H)=2$. By
Lemma \ref{n-}, $\mu_2(G^{n,a,b})<n-4,n-2$ and $\mu_{n-3}(G^{n,a,b})=n-2$, so the result follows.
\end{proof}

\begin{theorem}\label{d3}
Let $G$ be an $n$-vertex graph with diameter three, where $n\ge 5$. Then $m_G[n-2,n]\le n-3$ with equality if and only if $G\cong G_{n,3}$, $G_{n,3}^{(2,s)}$ with $1\le s\le n-4$, $G\cong G_{n,3}^{(4,s)}$ with $1\le s\le n-4$, $G^{n,a}$ with  $1\le a\le \frac{n}{2}-2$, or $G^{n,a,b}$ with $1\le a\le b$ and $a+b\le n-5$.
\end{theorem}

\begin{proof}
%As $n\ge 5$, it is easily seen that $G$ is a spanning subgraph of $G_{n,3}$ or $G_{n,3,a}$. By Lemmas \ref{addedges}, \ref{gn3} and \ref{g3a},  $\mu_{n-2}(G)<n-2$, so $m_G[n-2,n]\le n-3$.
%
%
If $G\cong G_{n,3}$, $G_{n,3}^{(2,s)}$ with $1\le s\le n-4$, $G\cong G_{n,3}^{(4,s)}$ with $1\le s\le n-4$, $G^{n,a}$ with  $1\le a\le \frac{n}{2}-2$, or $G^{n,a,b}$ with $1\le a\le b$ and $a+b\le n-5$, then we have by Lemmas \ref{h24} and \ref{g3ab} that $m_G[n-2,n]=n-3$.

Suppose the following that
$G\ncong G_{n,3}$, $G_{n,3}^{(2,s)}$ with $1\le s\le n-4$, $G\cong G_{n,3}^{(4,s)}$ with $1\le s\le n-4$, $G^{n,a}$ with  $1\le a\le \frac{n}{2}-2$, and $G^{n,a,b}$ with $1\le a\le b$ and $a+b\le n-5$. It suffices to show that  $m_G[n-2,n]<n-3$.

Let $P=v_1v_2v_3v_4$ be a diametral path of $G$. Assume that $\delta_G(v_1)\le \delta_G(v_4)$.

Suppose first that $\delta_G(v_1)=1$. Then $G$ is a spanning subgraph of $G_{n,3}$.
As $G\ncong G_{n,3}$, $G$ is a spanning subgraph of $G_{n,3}-e$ for some $e\in E(G_{n,3})$.
Let $G'=G_{n,3}$. If $e$ joins two vertices outside $P$, then $G'-e\cong G'-v_3u$ for some $u\in V(G)\setminus V(P)$. If $e=v_3v_i$ for $i=2,4$, then $G'-e\cong G'-v_iu$ for some $u\in V(G)\setminus V(P)$.
Thus, we may assume that $e$ joins a vertex outside $P$ and $v_i$ with $i=2,3,4$.
If $e= v_3u$ for $u\in V(G')\setminus V(P)$, then we have by  Lemmas \ref{addedges} and \ref{G123} that  $\mu_{n-3}(G)\le \mu_{n-3}(G'-e)<n-2$, so $m_G[n-2,n]<n-3$.
Assume that $e=v_2u$ or $e=v_4u$ for some $u$ outside $P$. Correspondingly, $G$ is a spanning subgraph of $G_{n,3}^{(2,1)}$
or $G_{n,3}^{(4,1)}$.
Note that $G\ncong G_{n,3}^{(2,s)},G_{n,3}^{(4,s)}$ for  $1\le s\le n-4$. So  $G$ is a spanning subgraph of $G_{n,3}^{(2,1)}-f$ for some edge $f$ not incident to $v_2$ or a spanning subgraph of $G_{n,3}^{(4,1)}-f$ for some edge $f$ not incident to $v_4$.
By similar argument as above, we may assume that $f$ joins $v_i$ and a vertex outside $P$ for $i=3,4$ in the former case and $i=2,3$ in the latter case.
Then
Lemmas \ref{addedges} and \ref{G123} imply that $\mu_{n-3}(G)\le \mu_{n-3}(G'-e-f)<n-2$, so $m_G[n-2,n]<n-3$.

Suppose next that $\delta_G(v_1)\ge 2$. Note that $v_1$ and $v_4$ share no common neighbors. So
$\delta_G(v_1)-1+\delta_G(v_4)-1\le n-4$, implying that $a:=\delta_G(v_1)-1\le \frac{n}{2}-2$.
It follows that  $G$ is a spanning subgraph of $G^{n,a}$.
As $G\ncong G^{n,a}$, $G$ is a spanning subgraph of $G^{n,a}-e$ for some $e\in E(G^{n,a})$.
Let $G^*=G^{n,a}$.
If both ends of $e=wz$ lie outside $P$ and  $z\in N_{G^*}(v_1)$, then $G^*-e\cong G^*-v_2w$.
If both ends of $e=wz$ lie outside $P$ and $z\in N_{G^*}(v_4)$, then  $G^*-e\cong G^*-v_3w$.
If $e=v_1v_2$, then $G^*-e\cong G^*-v_1u$ for some $u\in N_{G^*}(v_1)\setminus V(P)$. If $e=v_2v_3$, then $G^*-e\cong G^*-v_2u$ for some $u\in N_{G^*}(v_4)\setminus V(P)$. If $e=v_3v_4$, then $G^*-e\cong G^*-v_4u$ for some $u\in N_{G^*}(v_4)\setminus V(P)$.
So we may assume that $e$ joins a vertex outside $P$ and $v_i$ with $i=1,2,3,4$.

If $e$ is incident to $v_2$ or $v_3$, then Lemma \ref{gs} together with Lemma \ref{addedges}
implies $\mu_{n-3}(G)<n-2$, so $m_G[n-2,n]<n-3$.

Assume that  $e=v_1u$ or $e=v_4u$ for some $u$ outside $P$. Note that $\delta_G(v_1)\ge 2$. Then  $G^*-e\cong G^{n,a-1,n-4-a}$ with $a\ge 2$ or $G^*-e\cong G^{n,a,n-5-a}$. As $G\ncong G^{n,r,s}$ with $1\le r\le s$ and $r+s\le n-5$, $G$ is a spanning subgraph of $G^*-e-f$ for some $f\in E(G^{n,a})$, where
$f$ is incident to $v_2$ or $v_3$, or $f$ joins two vertices outside $P$. By similar argument as above, we may assume that $f$ joins $v_i$ with $i=2,3$ and a vertex $u$ outside $P$.
So $\mu_{n-3}(G)\le \mu_{n-3}(G^*-e-f)\le \mu_{n-3}(G^*-f)<n-2$ by Lemmas \ref{addedges} and \ref{gs}, implying that $m_G[n-2,n]<n-3$.
\end{proof}

\begin{proposition}\label{gn43}
$\mu_{n-3}(G_{n,4,3})>n-3$ and so $m_{G_{n,4,3}}[n-3,n]> n-4$.
\end{proposition}
\begin{proof}
Let $G=G_{n,4,3}$ and $P=v_1v_2v_3v_4v_5$ be the diametral path of $G$.
%Denote by $w$ and $z$ the vertex of degree $n-4$ in $G$.
It needs to prove $\mu_{n-3}(G)>n-3$ only.
As the rows of $(n-2)I_n-L(G)$ corresponding to vertices in $S:=V(G)\setminus\{v_1,v_2,v_4,v_5\}$ are equal, $n-2$ is a Laplacian eigenvalue of  $G$ with multiplicity at least $n-7$. Note that $L(G)$ has a quotient matrix $B$ with respect to the partition $V(G)=\{v_1\}\cup \{v_2\}\cup S\cup \{v_4\}\cup \{v_5\}$, where \[
B=\begin{pmatrix}
1&-1&0&0&0\\
-1&n-3&-(n-4)&0&0\\
0&-1&2&-1&0\\
0&0&-(n-4)&n-3&-1\\
0&0&0&-1&1
\end{pmatrix}.
\]
Let $\det(xI_5-B)=xf(x)$. As \[
f(n-1)=2n-5>0,\mbox{ }f(n-2)=-(n-4)^2<0\mbox{ and }f(n-3)=2n-9>0,
\]
$f(x)$ has roots $\alpha$ and $\beta$ with $n-2<\alpha<n-1$ and $n-3<\beta<n-2$. So by Lemma \ref{quo}, $\alpha$ and $\beta$ are Laplacian eigenvalues of $G$ and so $\mu_{n-3}(G)>n-3$, as desired.
\end{proof}

\begin{proposition}
For $d=4,\dots,n-3$ and $t=3,\dots,d-1$, $\mu_{n-d+1}(G_{n,d,t})>n-d+1$ and so $m_{G_{n,d,t}}[n-d+1,n]> n-d$.
\end{proposition}
\begin{proof}
By Proposition \ref{gn43}, it needs only to prove $\mu_{n-d+1}(G_{n,d,t})>n-d+1$ for $d\ge 5$.
Let $G'=G_{n,d,t}-\{v_iv_{i+1}:i=1,\dots,t-3,t+2,\dots,d \}$.
Then $G'=G_{n-d+4,4,3}\cup (d-4)K_1$.
By Lemma \ref{addedges} and Proposition \ref{gn43}, \[
\mu_{n-d+1}(G_{n,d,t})\ge \mu_{n-d+1}(G')=\mu_{n-d+1}(G_{n-d+4,4,3})>n-d+1,
\]
as desired.
\end{proof}

For integers  $n$ and $t$ with $1\le t\le n-3$, we denote by $P_{n,t}^{++}$ the graph obtained from $P_{n-1}=v_1\dots v_{n-1}$ by  adding a new vertex $u$ and two new edges $uv_t$ and $uv_{t+2}$.

\begin{theorem}\label{n-d+1}
Let $G$ be an $n$-vertex graph with a diameteral path $P=v_1\dots v_{d+1}$, where $d\le n-3$. If there are at least two vertices outside $P$ having at most two neighbors on $P$, then $m_G[n-d+1,n]\le n-d$.
\end{theorem}

\begin{proof}
Let $w,z$ be the two vertices outside $P$ with at most two neighbors on $P$. Let $H$ be the subgraph induced by  $V(P)\cup \{w,z\}$ and $B$  the principal submatrix of $L(G)$ corresponding to vertices of $H$.
If $\mu_4(H)<4$, then  by Lemmas \ref{interlacing} and \ref{cw},
\[
\mu_{n-d+1}(G)=\rho_{n-(d+3)+4}(L(G))\le \rho_4(B)\le \mu_4(H)+n-d-3<n-d+1,
\]
 so $m_G[n-d+1,n]\le n-d$. Thus, it suffices to show that  $\mu_4(H)<4$.

If $\delta_P(w)<2$ ($\delta_P(z)< 2$, respectively), $2-\delta_P(w)$ ($2-\delta_P(z)$, respectively) edges may be added to join $w$ ($z$, resepctively) and $2-\delta_P(w)$ ($2-\delta_P(z)$, respectively) vertices such that $w$ ($z$, respectively) is adjacent to exactly two vertices on $P$.
By Lemma \ref{addedges}, we may assume that $\delta_P(w)=\delta_P(z)=2$.
Let $v_p,v_q$ ($v_r,v_s$, respectively) be two neighbors of $w$ ($z$, respectively).
As $P$ is a diametral path, $w$ ($z$, respectively) is adjacent to at most three consecutive vertices and so $d_P(v_p,v_q)\le 2$ and $d_P(v_r,v_s)\le 2$.

\noindent
{\bf Case 1.} $\{v_p,v_q\}=\{v_r,v_s\}$.

Let $H'=H+wz$ if $wz\notin E(G)$ and $H'=H$ otherwise.
Then \[
L(H')=\begin{pmatrix}
L(P)&O_{(d+1)\times 2}\\
O_{2\times (d+1)}&L(P_2)
\end{pmatrix}+M,
\]
where $M=(m_{ij})_{(d+3)\times (d+3)}$ with \[
m_{ij}=\begin{cases}
2&\mbox{ if }i=j\in\{p,q,d+2,d+3\},\\
-1&\mbox{ if }\{i,j\}\in \{\{x,y\}:x=p,q,y=d+2,d+3\},\\
0&\mbox{ otherwise.}
\end{cases}
\]
Note that $M$ is  permutationally similar to $L(K_{2,2}\cup  (d-1)K_1)$, one gets
$\rho_4(M)=0$.
By Lemma \ref{cw}, we have \[
\mu_4(H')\le \rho_1(L(H')-M)+\rho_4(M)=\mu_1(P)<4,
\]
so $\mu_4(H)\le \mu_4(H')<4$ by Lemma \ref{addedges}.

\noindent
{\bf Case 2.} $|\{v_p,v_q\}\cap \{v_r,v_s\}|=1$.

Assume that $p<q=r<s$.
Suppose that $q-p=2$ or $s-r=2$, say $q-p=2$.
As the diameter of $G$ is $d$, $wz\notin E(G)$.
Then \[
L(H)=\begin{pmatrix}
L(P_{d+2,p}^{++})&O_{(d+2)\times 1}\\
O_{1\times (d+2)}&0
\end{pmatrix}+M,
\]
where  $M=(m_{ij})_{(d+3)\times (d+3)}$ with \[
m_{ij}=\begin{cases}
2&\mbox{ if }i=j=d+3,\\
1&\mbox{ if }i=j=r,s,\\
-1&\mbox{ if }\{i,j\}=\{r,d+3\},\{s,d+3\},\\
0&\mbox{ otherwise.}
\end{cases}
\]
By \cite[Lemma 4.3]{XZ},  $\mu_2(P_{d+2,p}^{++})<4$.
Note that $M$ is  permutationally similar to $L(P_3\cup  dK_1)$, one gets
$\rho_3(M)=0$.
By Lemma \ref{cw},
 \[
\mu_4(H)\le \rho_2(L(H)-M)+\rho_3(M)=\mu_2(P_{d+2,p}^{++})<4,
\]
as desired.
Suppose next that $q-p=s-r=1$.
Then $q=r=p+1$ and $s=p+2$.
Let $H'$ be the graph defined in Case 1. Then $H'-wz-v_pv_q-v_rv_s\cong P_{d+3}$.
Let $u_i=v_i$ for $i=1,\dots,p$, $u_{p+1}=w$, $u_{p+2}=v_q$, $u_{p+3}=z$ and $u_{i+2}=v_i$ for $i=s,\dots d+1$.
Then $L(H')=L(P_{d+3})+S$, where $S=(s_{ij})_{(d+3)\times (d+3)}$ with
\[
s_{ij}=\begin{cases}
2&\mbox{ if }i=j=p+2,\\
1&\mbox{ if }i=j=p,p+1,p+3,p+4,\\
-1&\mbox{ if }\{i,j\}=\{p,p+2\},\{p+2,p+4\},\{p+1,p+3\},\\
0&\mbox{ otherwise.}
\end{cases}
\]
As $S$ is  permutationally similar to $L(P_3\cup P_2\cup  (d-2)K_1)$, one gets
$\rho_4(S)=0$.
So by Lemma \ref{cw}, \[
\mu_4(H')\le \mu_1(P_{d+3})+\rho_4(S)<4,
\]
and therefore $\mu_4(H)\le \mu_4(H')<4$ by Lemma \ref{addedges},
as desired.

\noindent
{\bf Case 3.} $\{v_p,v_q\}\cap \{v_r,v_s\}=\emptyset$.

Assume that $p<q<r<s$.
If $q-p=2$ or $s-r=2$, then we have $\mu_4(H)<4$ by similar argument as Case 2.
Suppose next that $q-p=1$ and $s-r=1$.
Let $H'$ be defined as in Case 1.
%If $r-q\ge 2$, then $wz\notin E(G)$ as the diameter of $G$ is $d$.
Note that $H'-v_pv_q-v_rv_s-wz\cong P_{d+3}$. It hence follows from Lemma \ref{addedges} that \[
\mu_4(H)\le\mu_4(H')\le \mu_1(H'-v_pv_q-v_rv_s-wz)=\mu_1(P_{d+3})<4,
\]
as desired.
\end{proof}

By Theorems \ref{d2}, \ref{d3} and \ref{n-d+1}, all trees on $n$ vertices with diameter $d$  belong to $\mathfrak{G}(n,d)$ for $1\le d\le n-3$.

\vspace{5mm}

%\noindent {\bf Acknowledgement.}
%This work was supported by the National Natural Science Foundation of China (No.~12071158).


\begin{thebibliography}{99}

\bibitem{AhMS}
M. Ahanjideh, S. Akbari, M.H. Fakharan, V. Trevisan,
Laplacian eigenvalue distribution and graph parameters,
Linear Algebra Appl.
632 (2022) 1--14.

\bibitem{AM} W.N. Anderson, T.D. Morley,
Eigenvalues of the Laplacian of a graph,
Linear  Multilinear Algebra 18 (1985) 141--145.

\bibitem{BRT}
R.O. Braga, V.M. Rodrigues, V. Trevisan,
On the distribution of Laplacian eigenvalues of trees,
Discrete Math. 313 (2013) 2382--2389.


\bibitem{BH} A. Brouwer, W. Haemers, Spectra of Graphs, Springer, New York, 2012.

\bibitem{CDS} D. Cvetkovi\'c, M. Doob, H. Sachs,
Spectra of Graphs, 3rd edn., Johann Ambrosius
Barth Verlag, Heidelberg, 1995.




\bibitem{CJT} D.M. Cardoso, D.P. Jacobs, V. Trevisan,
Laplacian distribution and domination, Graphs Combin. 33  (2017) 1283--1295.

%
%\bibitem{HJ} R.A. Horn, C.R. Johnson,  Matrix Analysis, Second ed., Cambridge University Press, Cambridge, 2013.



\bibitem{GM}
R. Grone, R. Merris, The Laplacian spectrum of a graph II,
SIAM J. Discrete Math. 7 (1994) 221--229.

\bibitem{GMS} R. Grone, R. Merris, V.S. Sunder,
The Laplacian spectrum of a graph,
SIAM J. Matrix Anal. Appl. 11 (1990) 218--238.



\bibitem{GT}
J. Guo, S. Tan,
A relation between the matching number and Laplacian spectrum of a graph,
Linear Algebra Appl. 325 (2001) 71--74.

\bibitem{HJ} R.A. Horn, C.R. Johnson,  Matrix Analysis, Second ed., Cambridge University Press, Cambridge, 2013.
%\bibitem{GWZF}
%J. Guo, X. Wu, J. Zhang, K. Fang,
%On the distribution of Laplacian eigenvalues of a graph,
%Acta Math. Sin. (Engl. Ser.) 27 (2011) 2259--2268.

\bibitem{HJT} S.T. Hedetniemi, D.P. Jacobs, V. Trevisan,
Domination number and Laplacian eigenvalue distribution,
European J. Combin. 53 (2016) 66--71.

\bibitem{HJT} C. Hoppen, D.P. Jacobs, V. Trevisan,
Locating Eigenvalues in Fraphs: Algorithms and Applications,
Springer, Cham, 2022.

\bibitem{JOT} D.P. Jacobs,  E.R. Oliveira,  V. Trevisan,
Most Laplacian eigenvalues of a tree are small,
J. Combin. Theory Ser. B 146 (2021) 1--33.

%\bibitem{KT}  A. Knutson, T. Tao,
%Honeycombs and sums of Hermitian matrices,
%Notices Amer. Math. Soc. 48 (2001) 175--186.

\bibitem{Me}  R. Merris,  Laplacian matrices of graphs: a survey,
Linear Algebra Appl. 197--198 (1994) 143--176.


\bibitem{Me2} R. Merris,
The number of eigenvalues greater than two in the Laplacian spectrum of a graph,
Portugal. Math. 48 (1991) 345--349.

\bibitem{Moh} M. Mohar, The Laplacian spectrum of graphs, in: Y. Alavi, G. Chartrand, O.R. Oellermann, A.J. Schwenk, Graph theory, Combinatorics, and Applications, Vol. 2, Wiley, New York, 1991, pp. 871--898.

\bibitem{Mo} B. Mohar, Laplace eigenvalues of graphs--a survey, Discrete Math. 109 (1992) 171--183.

\bibitem{Sin} C. Sin,
On the number of Laplacian eigenvalues of trees less than the average degree,
Discrete Math. 343 (2020)  111986.

\bibitem{So} W. So,
Commutativity and spectra of Hermitian matrices,
Linear Algebra Appl. 212--213 (1994) 121--129.


\bibitem{TCDV}
V. Trevisan, J.B. Carvalho, R.R. Del Vecchio, C.T.M. Vinagre,
Laplacian energy of diameter 3 trees,
Appl. Math. Lett. 24 (2011)  918--923.

%\bibitem{We}
%H. Weyl, Hermann
%Das asymptotische Verteilungsgesetz der Eigenwerte linearer partieller Differentialgleichungen,
%Math. Ann. 71 (1912) 441--479 (in German).

\bibitem{ZZD} L. Zhou, B. Zhou, Z. Du,
On the number of Laplacian eigenvalues of trees smaller than two,
Taiwanese J. Math.  19  (2015) 65--75.



%
%
%\bibitem{GM}
%R. Grone, R. Merris, The Laplacian spectrum of a graph II,
%SIAM J. Discrete Math. 7 (1994) 221--229.
%
%\bibitem{HJ} R.A. Horn, C.R. Johnson,  Matrix Analysis, Second ed., Cambridge University Press, Cambridge, 2013.
%
%\bibitem{JOT} D.P. Jacobs,  E.R. Oliveira,  V. Trevisan,
%Most Laplacian eigenvalues of a tree are small,
%J. Combin. Theory Ser. B 146 (2021) 1--33.
%
%\bibitem{Me}
%R. Merris,  Laplacian matrices of graphs: a survey, Linear Algebra Appl. 197--198 (1994) 143--176.
%
%\bibitem{Mo} M. Mohar, The Laplacian spectrum of graphs, in: Y. Alavi, G. Chartrand, O.R. Oellermann, A. J. Schwenk, Graph theory,  Combinatorics, and Applications, Vol. 2, Willey, New York, 1991, pp. 871--898.
%
%\bibitem{Moh} M. Mohar, The Laplacian spectrum of graphs, in: Y. Alavi, G. Chartrand, O.R. Oellermann, A.J. Schwenk, Graph theory, Combinatorics, and Applications, Vol. 2, Wiley, New York, 1991, pp. 871--898.
%
%\bibitem{So} W. So,
%Commutativity and spectra of Hermitian matrices,
%Linear Algebra Appl. 212--213 (1994) 121--129.
%
\bibitem{XZ}
L. Xu, B. Zhou, Proof of a conjecture on distribution of Laplacian eigenvalues and diameter, and beyond, arXiv:2303.11503.

\end{thebibliography}
\end{document}